\documentclass[11pt,twoside]{amsart}

\usepackage{amssymb,amsthm}
\usepackage{latexsym}
\usepackage[english]{babel}
\theoremstyle{plain}
\newtheorem{thm}{Theorem}[section]

\newtheorem{prop}[thm]{Proposition}
\newtheorem{cor}[thm]{Corollary}
\newtheorem{lemma}[thm]{Lemma}

\newtheorem{prob}[thm]{Problem}

\theoremstyle{definition}

\begin{document}

\title[On the mixed derivatives of a separately twice differentiable function]
{On the mixed derivatives of a separately twice differentiable function}

\author{V.~Mykhaylyuk}

\address{Department of Mathematics and Informatics, Chernivtsi National University,
Chernivtsi, Ukraine}

\email{vmykhaylyuk@ukr.net}




\begin{abstract} WE prove that a function $f(x,y)$ of
real variables defined on a rectangle, having square integrable partial derivatives
$f''_{xx}$ and $f''_{yy}$, has almost everywhere mixed derivatives $f''_{xy}$
and $f''_{yx}$.
\end{abstract}
\maketitle

\section{Introduction}

In the well known ``Scottish Book'' \cite{Mauldin} S.~Mazur posed the
following question (VII.1935, Problem 66):

{\it The real function $z=f(x,y)$ of real variables $x,y$ possesses the
1st partial derivatives $f'_x$, $f'_y$ and the pure second partial derivatives
$f''_{xx}$, $f''_{yy}$. Do there exist then almost everywhere the mixed 2nd partial
derivatives $f''_{xy}$, $f''_{yx}$? According to a remark by p. Schauder, this
theorem is true with the following additional assumptions: The derivatives
$f'_x$, $f'_y$ are absolutely continuous in the sense of Tonelli, and the derivatives
$f''_{xx}$, $f''_{yy}$ are square integrable. An analogous question for
$n$ variables.}

The existence and measurability of (mixed) partial derivatives were investigated by many mathematicians (see \cite[Th. 79-81]{Bernstein0}, \cite{H-D}, \cite{Cur}, \cite{Tolstov}, \cite{Tolstov1}, \cite{Mitiagin}, \cite{Bugrov}, \cite{Ser}, \cite{MM}, \cite{Min} and the literature given there). Mainly, these results give some sufficient conditions for existence (and equality) almost everywhere of mixed second partial derivatives, but they do not give any answer to the Mazur problem.
In particular, G.~Tolstov in \cite{Tolstov1} proved the following result (see also \cite[Lemma 4]{Min}).

\begin{prop} \label{pr:1.1}
Let $h:[0,1]^2\to\mathbb R$ be an integrable function and
$$
f(x,y)=\int_0^xdx\int_0^yh(u,v)dudv.
$$
Then there exists a measurable set $A\subseteq [0,1]$ with $\mu(A)=1$ such that
$$
f'_x(x_0,y)=\int_0^yh(u,v)dv
$$
for every $x_0\in A$ and $y\in[0,1]$.
\end{prop}
Using this statement G.~Tolstov proved that if a separately differentiable function $f:[0,1]^2\to\mathbb R$ has a jointly continuous $f'_x$ and there exists $f''_{xy}$ on a set $D\subseteq [0,1]^2$ of the second category, then there exists a rectangle $P\subseteq [0,1]^2$ such that $f''_{xy}=f''_{yx}$ almost everywhere on $P$. This result was developed in \cite[Theorem 7]{Min}. Moreover, in \cite[Theorem 7]{Tolstov} G.~Tolstov proved the following theorem.

\begin{thm}\label{th:1.2}
Let $f:[0,1]^2\to\mathbb R$ and $f'_{x}$,
$\,f'_{y}$ have all finite derivative numbers with respect to each variable
on some subset $E\subseteq P$ of positive measure. Then a.e. on $E$ there
exist and are equal the mixed derivatives $f''_{xy}$ and $f''_{yx}$.
\end{thm}

On the other hand, there are some interesting examples among the above mentioned results. G.~Tolstov was constructed in \cite{Tolstov1} a function $f:[0,1]^2\to\mathbb R$ having jointly continuous first partial derivatives and mixed second partial
derivatives  which are different on a set of positive measure. Moreover, J.Serrin was constructed in \cite{Ser} a measurable function $f:[0,1]^2\to\mathbb R$
which is a.e. differentiable on each horizontal segment as a function of one variable, but for which the set of the existence of partial derivative
with respect to the first variable is non-measurable.

The Mazur problem was solved in the negative in \cite{MP}. It was constructed in \cite{MP} a separately twice differentiable
function $f:[0,1]^2\to\mathbb R$, whose partial derivative $f'_x$ is discontinuous with respect to the second
variable on a set $A\times B\subseteq [0,1]^2$ with $\mu(A)=1$ and $\mu(B)>0$. This example shows that for a separately twice differentiable
function $f:[0,1]^2\to\mathbb R$ the continuity of $f'_x$ with respect to the second variable plays an important role for the existence of $f''_{xx}$.

Note that the second partial derivatives $f''_{xx}$, $f''_{yy}$ of the function $f$ from \cite{MP} are not integrable.
Thus the following question naturally arises in the connection with Schauder's remark to the Mazur problem and the example from \cite{MP}.

\begin{prob}\label{prob:1.1}
Let $f:[0,1]^2\to\mathbb R$ be a separately twice differentiable function and
its second derivatives $f''_{xx}$ and $f''_{yy}$ are square integrable.
\begin{enumerate}
\item[$(i)$] Does there exist a set $A\subseteq[0,1]$ with $\mu(A)=1$ such
that $f'_x$ is continuous with respect to $y$ at each point
of $A\times[0,1]$?
\item[$(ii)$] Do there exist almost everywhere mixed derivatives $f''_{xy}$ and $f''_{yx}$?
\end{enumerate}
\end{prob}

In this paper we give the positive answer to Problem \ref{prob:1.1}. In Section 2 (Corollary \ref{pr:3.3}) we give some sufficient conditions on a function $f(x,y)$ for the jointly continuity of the first partial derivative $f'_x$. In Section 3 we prove an auxiliary statement (Proposition \ref{pr:4.3}) on the consistency of the Fourier series of a function $f$ and its partial derivative $f'_x$, which we use in Section 4 for the proof of the main result of the paper (Theorem \ref{th:4.51}). Finally, in Section 5 we give two examples which show the essentiality of some assumptions in Corollary \ref{pr:3.3} and Theorem \ref{th:4.51} and formulate open questions.

\section{Jointly continuity of the first partial derivative}

\begin{lemma} \label{lem:3.1}
Let $Y\subseteq\mathbb{R}$, a function $f:[0,1]\times Y\to\mathbb R$ be
continuous with respect to $y$ and $f'_x$ be (uniformly) continuous with
respect to $x$, uniformly on $y$. Then $f'_x$ is jointly continuous.
\end{lemma}
\begin{proof}
Fix $x_0\in [0,1]\,$, $\,y_0\in Y$ and $\varepsilon>0$. Choose a neighborhood
$U=[x_1,x_2]$ of $x_0$ such that for all $ x_,x'\in U$ and all $y$
$$|f'_x(x,y)-f'_x(x',y)|<\frac{\varepsilon}{4}\,.$$

Using the continuity of $f$ with respect to $y$, choose a neighborhood $V$ of
$y_0$ such that for all $y\in V$
$$\left|\frac{f(x_2,y_0)-f(x_1,y_0)}{x_2-x_1}-\frac{f(x_2,y)-f(x_1,y)}{x_2-x_1}\right|<\frac{\varepsilon}{2}\,.$$
By the Lagrange theorem, for every $y\in V$ there is
$x_y\in U$ such that
$$f'_x(x_y,y)=\frac{f(x_2,y)-f(x_1,y)}{x_2-x_1}.$$

Hence, for each $y\in V$
$$|f'_x(x_{y_0},y_0)-f'_x(x_{y},y)|<\frac{\varepsilon}{2}\,.$$

Therefore, for an arbitrary $(x,y)\in U\times V$
$$|f'_x(x_0,y_0)-f'_x(x,y)|\leq $$
$$|f'_x(x_0,y_0)-f'_x(x_{y_0},y_0)|+|f'_x(x_{y_0},y_0)-f'_x(x_{y},y)|+
|f'_x(x_y,y)-f'_x(x,y)|<$$
$$\frac{\varepsilon}{4}+\frac{\varepsilon}{4}+\frac{\varepsilon}{2}=\varepsilon.$$
\end{proof}
\begin{lemma} \label{lem:3.2}
Let $g:[0,1]\to\mathbb{R}\,$ and $\;\int_0^1 (g')^2dx\le c$.
Then the function $g$ is H\"older of the order $\frac{1}{2}$ and constant $\sqrt{c}$.
\end{lemma}

\begin{proof}
For all $0\le x_0<x_1\le 1$
$$|g(x_1)-g(x_0)|=\left|\int_{x_0}^{x_1} g'dx\right|\leq
\left((x_1-x_0)\int_{x_0}^{x_1}(g')^2dx\right)^{\frac{1}{2}}\leq
\sqrt{c(x_1-x_0)}\,.$$
\end{proof}

The next theorem is a simple combination of the previous two lemmas.

\begin{cor} \label{pr:3.3}
Let $Y\subseteq\mathbb{R}$, a function $f:[0,1]\times Y\to\mathbb R$ be
continuous with respect to $y$, there exists $f''_{xx}$ and
$$\sup_{y\in Y}\int_0^1 (f''_{xx}(x,y))^2dx<\infty.$$

Then $f'_x$ is jointly continuous.
\end{cor}

\section{Square integrable partial derivatives}

\begin{lemma} \label{pr:4.1}
Let $\int_0^{2\pi}g^2(x)dx<\infty\,$, $\,\int_0^{2\pi}g(x)dx=0$ and
$f(x)=a+\int_0^{x}g(t)d\mu$ $($in particular let $f$ be a
differentiable function such that $f(0)=f(2\pi)$ and the
derivative $g=f'$ be absolutely integrable on $[0,2\pi]$$)$. Then for every $n\in\mathbb N$
$$\int_0^{2\pi}f(x)\cos nxdx=-n\int_0^{2\pi}g(x)\sin nxdx$$ and
$$\int_0^{2\pi}f(x)\sin nxdx=n\int_0^{2\pi}g(x)\cos nxdx.$$
\end{lemma}

\begin{proof}
By \cite[p.~251]{Nat}, for a differentiable function $f$ we
have $f(x)=f(0)+\int_0^{x}f'(t)d\mu$ for every $x\in [0,2\pi]$. It remains
to use the integration by parts \cite[Ch.~IX, \S8, Th.~5]{Nat}.
\end{proof}

For a function $f:[0,2\pi]^2\to \mathbb{C}$, the expression
$$f\sim\sum\nolimits_{n,m\in\mathbb Z} a_{nm}e^{inx}e^{imy}$$
will denote that $f$ is square integrable and $\sum\nolimits_{n,m\in\mathbb Z} a_{nm}e^{inx}e^{imy}$
is the Fourier series of $f$ which converges to $f$ in the $L_2$-norm. The same
concerns a function $f:[0,2\pi]\to \mathbb{C}$.

\begin{lemma} \label{pr:4.2}
Let $\;f\sim\sum\nolimits_{n,m\in\mathbb Z} a_{nm}e^{inx}e^{imy}\,$,
$\;\alpha_n\sim\sum\nolimits_{m\in\mathbb Z} a_{nm}e^{imy}$, $n\in\mathbb Z$
and $f_y(x):=f(x,y)$.

Then there exists a subset $B\subseteq [0,2\pi]$ with $\mu(B)=2\pi$ such that
$\forall\,y\in B$ the function $f_y$ is square integrable and
\begin{equation}\label{eq:4}
f_y\sim\sum\nolimits_{n\in\mathbb Z} \alpha_n(y)e^{inx}.
\end{equation}
\end{lemma}

\begin{proof} For every $n\in\mathbb Z$ we consider the linear continuous operator $T_{1,n}:L_2([0,2\pi]^2)\to L_2[0,2\pi]$, which any function $g\sim\sum\nolimits_{n,m\in\mathbb Z} b_{nm}e^{inx}e^{imy}$ sends to the function $T_{1,n}g\sim\sum\nolimits_{m\in\mathbb Z} b_{nm}e^{imy}$. Note that $T_{1,n}f=\alpha_n$ for every $n\in\mathbb Z$. Besides we consider the linear operator $T_{2,n}:L_2([0,2\pi]^2)\to L_2[0,2\pi]$, $$T_{2,n}g(y)=\frac{1}{2\pi}\int_{0}^{2\pi}g(x,y)e^{-inx}dx.$$ Since $$\left(\int_{0}^{2\pi}h(x)dx\right)^2\leq 2\pi\int_{0}^{2\pi}|h(x)|^2dx$$ for every measurable on$[0,2\pi]$ function $h$, by the Fubini theorem we have
$$
\int_{0}^{2\pi}T_{2,n}g^2(y)dy\leq \frac{1}{2\pi}\int_{0}^{2\pi}\int_{0}^{2\pi}|g(x,y)e^{-inx}|^2dxdy=$$
$$=\frac{1}{2\pi}\int_{0}^{2\pi}\int_{0}^{2\pi}|g(x,y)|^2dxdy.
$$
Thus $T_{2,n}$ is continuous. Since $T_{1,n}(e^{ikx}e^{imy})=T_{2,n}(e^{ikx}e^{imy})$ for every $k,m\in\mathbb Z$, $T_{1,n}=T_{2,n}$, in particular, $T_{2,n}f=\alpha_n$. By the Fubini theorem there exists a set $B_1\subseteq [0,2\pi]$ with $\mu(B_1)=2\pi$ such that $f_y\in L_2[0,2\pi]$ for all $y\in B_1$. Now we choose a set $B\subseteq B_1$
with $\mu(B)=2\pi$ such that $\alpha_n(y)=\frac{1}{2\pi}\int_{0}^{2\pi}f(x,y)e^{-inx}dx$  for every $n\in\mathbb Z$ and $y\in B$.
This gives (\ref{eq:4}).
\end{proof}

\begin{prop} \label{pr:4.3}
Let a function $f:[0,2\pi]^2\to\mathbb R$ be differentiable with
respect to $x$, $f'_x$ be square integrable and, moreover,
$f(0,y)=f(2\pi,y)=\alpha(y)$ for every $y\in [0,2\pi]$ with the square integrable
$\alpha(y)$. Then $f$ is also square integrable. Moreover, if
$$f\sim\sum\nolimits_{n,m\in\mathbb Z} a_{nm}e^{inx}e^{imy}\;\;\;\;\;\;\mbox{then}
\;\;\;\;\;\;f'_x\sim\sum\nolimits_{n,m\in\mathbb Z} ina_{nm}e^{inx}e^{imy}.$$
\end{prop}

\begin{proof}
By the Fubini theorem there exists a set $B\subseteq [0,2\pi]$ with $\mu(B)=2\pi$
such that for every $y\in B$ the function $g_y:[0,2\pi]\to\mathbb R\,$, $\,g_y(x)=f'_x(x,y)$ is square integrable, and in particular integrable. Note that for
an arbitrary $y\in B$ we have
$$
\int_0^{2\pi}f'_x(x,y)dx=f(2\pi,y)-f(0,y)=0.
$$
So, by the Fubini theorem, $$b_{0m}=\frac{1}{4\pi^2}\int_0^{2\pi}\int_0^{2\pi}f'_x(x,y)e^{-imy}dxdy=0$$ for all $m\in\mathbb Z$.

Consider the function $h(x,y)=f(x,y)-\alpha(y)$. For every $y\in B$ and $x\in[0,2\pi]$ we have
$$
h^2(x,y)=\left(\int_0^xg_y(t)dt\right)^2\leq x \int_0^xg^2_y(t)dt\leq 2\pi \int_0^{2\pi}g^2_y(t)dt.
$$
Thus
$$
\int_0^{2\pi}\int_0^{2\pi}h^2(x,y)dxdy\leq 4\pi^2 \int_0^{2\pi}\int_0^{2\pi}(f'_x(x,y))^2dxdy,
$$
and $h$ is square integrable. So $f$ is square integrable too. Let $$f\sim\sum\nolimits_{n,m\in\mathbb Z} a_{nm}e^{inx}e^{imy}.$$

Now using Lemma \ref{pr:4.1} and the Fubini theorem for every $m, n\in\mathbb Z$, $n\ne 0$, we have
$$
a_{nm}=\frac{1}{4\pi^2}\int_B e^{-imy}dy\int_0^{2\pi}f(x,y)e^{-inx}dx=
$$
$$
=\frac{1}{in}\frac{1}{4\pi^2}\int_B e^{-imy}dy\int_0^{2\pi}f'_x(x,y)e^{-inx}dx= \frac{b_{nm}}{in}.
$$

\end{proof}

\section{Main result}

The following theorem gives a positive answer to
the Mazur problem for functions with square integrable pure partial derivatives (Problem \ref{prob:1.1}).

\begin{thm}\label{th:4.51}
Let $f:[0,1]^2\to\mathbb R$ and there exist square integrable derivatives
$f''_{xx}\,,\;f''_{yy}$. Then

$(i)$
a.e. there are equal mixed derivatives $f''_{xy}$ and $f''_{yx}$ which are square integrable;

$(ii)$ there exists $A\subseteq[0,1]$ with $\mu(A)=1$ such
that $f'_x$ is continuous with respect to $y$ at every point
of $A\times[0,1]$;

$(iii)$ $f$ is jointly continuous.

\end{thm}

\begin{proof} $(i)$. We consider a function $f:[0,2\pi]^2\to \mathbb{R}$. It is sufficient to prove this assertion for a sequence of products
$f\cdot\varphi_n$ of $f$ by twice differentiable functions $\varphi_n$ with
bounded derivatives and the conditions $\varphi_n(x,y)=1$ for all
$(x,y)\in[\frac{1}{n},2\pi-\frac{1}{n}]^2$ and $\varphi_n(x,y)=0$
on an open in $[0,2\pi]^2$ set which contains the boundary of $[0,2\pi]^2$. Thus it is sufficient to consider a function $f$ which satisfies the additional assumption
$$f(x,y)=f'_x(x,y)=f'_y(x,y)=f''_{xx}(x,y)=f''_{yy}(x,y)=0$$
on an open in $[0,2\pi]^2$ set which contains the boundary of $[0,2\pi]^2$.

By Proposition \ref{pr:4.3}, the function $f$, $f'_x$ and $f'_y$ are square integrable.
 Let
$$f\sim\sum\nolimits_{n,m\in\mathbb Z}a_{nm}e^{inx}e^{imy}.$$
Then, by Proposition \ref{pr:4.3},
$$f'_x\sim\sum\nolimits_{n,m\in\mathbb Z}ina_{nm}e^{inx}e^{imy},\quad f'_y\sim\sum\nolimits_{n,m\in\mathbb Z}ima_{nm}e^{inx}e^{imy},$$
$$f''_{xx}\sim\sum\nolimits_{n,m\in\mathbb Z}-n^2a_{nm}e^{inx}e^{imy}\quad {\mbox and}\quad f''_{yy}\sim\sum\nolimits_{n,m\in\mathbb Z}-m^2a_{nm}e^{inx}e^{imy}.$$
Let
$$\alpha_m\sim\sum\nolimits_{n\in\mathbb Z}ina_{nm}e^{inx},\;m\in\mathbb Z.$$
Using Proposition \ref{pr:4.2}, we choose a set $A_1\subseteq [0,2\pi]$ so that
$\mu(A_1)=2\pi$ and for every $x\in A_1\,$ the function $g^x$ is square integrable and
$\;g^x\sim\sum\nolimits_{m\in\mathbb Z}\alpha_m(x)e^{imy}$, where
$g^x(y):=f'_x(x,y)$. Since there exist open neighborhoods $V_1$ and $V_2$ of
points $0$ and $2\pi$ in $[0,2\pi]$ such that $g^x(y)=0$ for every $x\in A_1$  and $y\in V_1\cup V_2$, according to the well-known
localization theorem of Riemann we have $$\sum\nolimits_{m\in\mathbb Z} \alpha_m(x)=0$$
for every $x\in A_1$. Since
$$\sum\nolimits_{n,m\in\mathbb Z}m^2n^2 |a_{nm}|^2\leq
\sum\nolimits_{n,m\in\mathbb Z}(m^4+n^4) |a_{nm}|^2<\infty,$$
there exists a square integrable function $h:[0,2\pi]^2\to\mathbb R$ with
$$h\sim-\sum\nolimits_{n,m\in\mathbb Z}m n a_{nm}e^{inx}e^{imy}.$$
Once more, using Proposition \ref{pr:4.2} we choose a set
$A_2\subseteq A_1$ so that $\mu(A_2)=2\pi$ and for every $x\in A_2\,$,
the function $h^x$ is square integrable and $\;h^x\sim\sum\nolimits_{m\in\mathbb Z} im\alpha_m(x)e^{imy}$, where
$h^x(y):=h(x,y)$. Put
$$F(x,y)=\int_{0}^{y}h(x,t)dt.$$
Using Theorem 3 of \cite[Ch.~X, \S4]{Nat} on the termwise integration of Fourier series
of absolutely integrable functions, we obtain that for every $x\in A_2$ the equality
$$F(x,y)=\sum\nolimits_{m\in\mathbb Z} \alpha_m(x)(e^{imy}-1)=
\sum\nolimits_{m\in\mathbb Z} \alpha_m(x)e^{imy}$$
is satisfied. Note that the function $F$ is square integrable (it may be obtained analogously as the square integrability of $h$ in the proof of the Proposition \ref{pr:4.3}). Thus $F=f'_x$ in $L_2[0,2\pi]^2$.

Hence, there exists a set $B\subseteq [0,2\pi]$ with $\mu(B)=2\pi$ such that
$$\mu\{x\in[0,2\pi]:f'_x(x,y)=F(x,y)\}=2\pi$$
and the function $g_y(x)=f'_x(x,y)$ is integrable for every $y\in B$.

Consider the function $$G(x,y)=\int_{0}^{x}du\int_{0}^{y}h(u,v)dv.$$ Now for every $x\in[0,2\pi]$ and $y\in B$ we have
$$
f(x,y)=\int_0^{x}f'_x(u,y)du=\int_0^{x}F(u,y)du=\int_{0}^{x}du\int_{0}^{y}h(u,v)dv=G(x,y).
$$
Since $f$ is continuous with respect to the second variable, $G$ is jointly continuous and
$B$ is dense in $[0,2\pi]$, $f(x,y)=G(x,y)$ for every $(x,y)\in[0,2\pi]^2$. According to $\S 7$ from \cite{Tolstov1}, a.e. there exist mixed derivatives $f''_{xy}$ and $f''_{yx}$ which are a.e. equal to $h$ and hence are square integrable.

$(ii)$, $(iii)$. It follows from the proof of $(i)$ that for every $n\in\mathbb N$ there exists a function $G:[0,1]^2\to\mathbb R$, $G(x,y)=\int_{0}^{x}du\int_{0}^{y}h(u,v)dv$,
such that $f(x,y)=G(x,y)$ for every $(x,y)\in[\frac{1}{n},1-\frac{1}{n}]^2$. Therefore $f$ is jointly continuous on $(0,1)^2$ and according to $\S 7$ from \cite{Tolstov1} there exists
$A_0\subseteq[0,1]$ with $\mu(A_0)=1$ such
that $f'_x$ is continuous with respect to $y$ at every point
of $A\times(0,1)$. It remains to use this fact to some separately twice differentiable extension $\tilde{f}:(-1,2)^2\to\mathbb R$ of $f$
with square integrable derivatives
$\tilde{f}''_{xx}$ and $\tilde{f}''_{yy}$.
\end{proof}

\begin{cor}\label{cor:4.6}
Let $f(x,y)$ have on $[0,1]^2$ the second pure partial derivatives. Then there
exists an open dense set $G\subseteq [0,1]^2$ on which there are equal mixed
partial derivatives $f''_{xy}$ and $f''_{yx}$.
\end{cor}

\begin{proof}
Note that by \cite[p.~427]{Tolstov} the functions $f''_{xx}$ and $f''_{yy}$ are of
the first Baire class. Hence, there exists an open dense subset $G\subseteq [0,1]^2$
on which the pure derivations are locally bounded. It remains to use Theorem \ref{th:4.51}.
\end{proof}

\section{Examples, questions}

For a real valued function $f$, we denote $\mathrm{supp}f=\{x\in \mathbb{R}:f(x)\ne 0\}$.

The following example shows that the assumption $$\sup_{y\in Y}\int_0^1 (f''_{xx}(x,y))^2dx<\infty$$ in Corollary \ref{pr:3.3} cannot be replaced by
$$\sup_{y\in Y}\int_0^1 |f''_{xx}(x,y)|dx<\infty.$$

\begin{thm} \label{th:4.7}
There exists a function $f:[0,1]^2\to\mathbb R$ satisfying the following conditions:
\begin{enumerate}
\item[$(i)$] $f$ is separately infinitely differentiable;
\item[$(ii)$] $\sup_{y\in [0,1]}\int_0^1 |f''_{xx}(x,y)|dx=\sup_{x\in [0,1]}\int_0^1 |f''_{yy}(x,y)|dy<\infty$;
\item[$(iii)$] $f'_x$ and $f'_y$ are jointly discontinuous at every point of some closed set $E$ of positive measure.
\end{enumerate}
\end{thm}
\begin{proof}
Let
$$C=[0,1]\setminus \bigsqcup_{n=1}^{\infty}\bigsqcup_{k=1}^{2^{n-1}}(a_{n,k},b_{n,k})$$
be a Cantor type set of positive measure such that

$(1)$ $0<a_{n,k}<b_{n,k}<1$ for every $n$ and $k$;

$(2)$ $a_{n,k}\ne b_{m,l}$ for every $n,m\in\mathbb N$, $k\leq 2^{n-1}$ and $l\leq 2^{m-1}$;

$(3)$ $b_{n,k}-a_{n,k}=b_{n,l}-a_{n,l}$ for every $n\in\mathbb N$ and $k,l\leq 2^{n-1}$.

Let $\{a_{n,k}, b_{n,k}: n\in\mathbb N, 1\leq k\leq 2^{n-1}\}=\{p_n:n\in\mathbb N\}$, $\varphi:\mathbb N\to\mathbb N^3$ be a bijection. Inductively for $n$  we choose
a sequence $(W_n)_{n=1}^{\infty}$ of rectangle $W_n=U_n\times V_n$ such that

$(a)$ $U_n=(a_n,b_n), V_n=(c_n,d_n)\in \{(a_{m,k},b_{m,k}):m\in\mathbb N, 1\leq k\leq 2^{m-1}\}$;

$(b)$ $U_n\cap U_m=V_n\cap V_m=\emptyset$ for all distinct $n,m\in\mathbb N$;

$(c)$ $b_n-a_n=d_n-c_n$ for all $n\in\mathbb N$;

$(d)$ $W_n\subseteq \{(x,y)\in\mathbb R^2:\max\{|x-c_k|,|y-c_m|\}\leq\frac{1}{l}\}$ for every $n\in\mathbb N$, where $(k,m,l)=\varphi(n)$.

Note that $E=C^2\subseteq \overline{\{w_n:n\in\mathbb N\}}$ for every sequence $(w_n)_{n=1}^{\infty}$ of points $w_n\in W_n$.

Let $\psi:\mathbb R\to\mathbb{R}^+$
be an arbitrary infinitely differentiable function with
$\mathrm{supp}\,\psi(y)=(0,1)$ and $\max_{x\in[0,1]}|\psi(x)|=1$. For every $n\in\mathbb N$ we put $\varepsilon_n=b_n-a_n=d_n-c_n$,
$$\varphi_n(x)=\psi\left(\frac{x-a_n}{\varepsilon_n}\right)$$
and
$$\psi_n(y)=\psi\left(\frac{y-c_n}{\varepsilon_n}\right).$$

Consider the function
$f:[0,1]^2\to\mathbb{R}$,
$$f(x,y)=\sum\nolimits_{n=1}^{\infty}\varepsilon_n\varphi_n(x)\psi_n(y).$$

It follows from $(b)$ that $f$ is separately infinitely differentiable. Clearly,
$$\sup_{y\in [0,1]}\int_0^1 |f''_{xx}(x,y)|dx=\sup_{x\in [0,1]}\int_0^1 |f''_{yy}(x,y)|dy=\int_0^1 |\psi''(x)|dx.$$
Thus $f$ satisfies the condition $(ii)$.

We show that $f$ satisfies the condition $(iii)$. For every $n\in\mathbb N$ we choose $u_n\in U_n$ and $v_n\in V_n$ such that
$$\varphi_n(u_n)=A:=\max_{x\in[0,1]}|\psi'(x)|\quad{\mbox and}\quad |\psi_n(v_n)|=1.$$
Therefore $|f'_{x}(u_n,v_n)|=A$ for every $n\in\mathbb N$. Since $f'_{x}(z)=0$ for every $z\in E$ and $E=\subseteq \overline{\{(u_n,v_n):n\in\mathbb N\}}$,
$f'_{x}$ is jointly discontinuous at every point of $E$.

Analogously $f'_{y}$ is jointly discontinuous at every point of $E$.
\end{proof}

The following modification of the example from \cite[Theorem 3.2]{MP}
shows that the assumptions of the existence of $f''_{yy}$ and
$f''_{xx}$ everywhere on the rectangle $[0,2\pi]^2$ in Theorem \ref{th:4.51} cannot be weakened.

\begin{thm} \label{th:4.6}
There exists a function $f:[0,1]^2\to\mathbb R$ satisfying the following conditions:
\begin{enumerate}
\item[$(i)$] $f$ has continuous partial derivative $f''_{yy}$;
\item[$(ii)$] for every $y\in [0,1]$ there exists a finite set $A(y)$ such that
$f''_{xx}(x,y)=0$ for all $x\in [0,1]\setminus A(y)$;
\item[$(iii)$] the set $\bigcup_{y\in[0,1]}A(y)$ is countable;
\item[$(iv)$] $f'_x$ is discontinuous with respect to $y$
at every point of some set $E$ of positive measure, in particular $f''_{xy}$
does not exist at all points of $E$.
\end{enumerate}
\end{thm}

\begin{proof}
We construct a function $f$ similarly as in the proof of Theorem 3.2 from \cite{MP}, modifying the functions $\varphi_n$ only.

Let $B\subset [0,1]$ be a closed set without isolated points with
$\mu(B)>0$, whose complement $[0,1]\setminus B$ is dense in $[0,1]$ and $$[0,1]\setminus B=\bigsqcup\limits_{n=1}^{\infty}(a_n,b_n).$$  Let $\psi:\mathbb R\to\mathbb{R}^+$
be an arbitrary twice differentiable function with
$\mathrm{supp}\,\psi(y)=(0,1)$, $$\psi_n(y)=\psi\left(\frac{y-a_n}{b_n-a_n}\right)\,,
\;\;n=1,2,\dots$$
$\varepsilon_n>0$ so that
$\lim\limits_{n\to\infty}\frac{\varepsilon_n}{(b_n-a_n)^2}=0$.

We choose continuous functions $\varphi_n:[0,1]\to [0,\varepsilon_n]$ so that $|\varphi'_n(x)|=1$ for all $x\in [0,1]\setminus A_n$,
where $A_n$ is some finite set.

The function $f:[0,1]^2\to\mathbb{R}$,
$$f(x,y)=\sum\nolimits_{n=1}^{\infty}\varphi_n(x)\psi_n(y),$$

satisfies conditions $(i)-(iii)$ and condition $(iv)$ for $$E=\left([0,1]\setminus\bigcup_{n=1}^\infty A_n\right)\times B.$$
\end{proof}

In connection with this example and Theorem \ref{th:4.51}, the following question naturally
arises.

\begin{prob}\label{prob:5.4}
Let $f:[0,1]^2\to\mathbb R$ be a separately twice differentiable function and
its second derivatives $f''_{xx}$ and $f''_{yy}$ be integrable.
\begin{enumerate}
\item[$(i)$] Do there exist a.e. the mixed derivatives $f''_{xy}$ and $f''_{yx}$?
\item[$(ii)$] Does there exist a set $A\subseteq[0,1]$ with $\mu(A)=1$ such
that $f'_x$ is continuous with respect to $y$ in each point
of $A\times[0,1]$?
\item[$(iii)$] Is the function $f$ jointly continuous?
\end{enumerate}
\end{prob}

It follows from the proof of Theorem \ref{th:4.51} that the conditions of square integrability of $f''_{xx}$ and $f''_{yy}$ can be replaced by the
integrability of $f''_{xx}$ and $f''_{yy}$ and by the existence of an integrable function $h:[0,2\pi]^2\to\mathbb R$ with
$$h\sim-\sum\nolimits_{n,m\in\mathbb Z}m n a_{nm}e^{inx}e^{imy}.$$
In this connection the following question naturally
arises.

\begin{prob}\label{prob:5.5}
Let $f:[0,2\pi]^2\to\mathbb R$ be a separately twice differentiable function  with $f(x,y)=0$
on an open set in $[0,2\pi]^2$ which contains the boundary of $[0,2\pi]^2$ and
its second derivatives $f''_{xx}$ and $f''_{yy}$ are integrable. Let
$$f\sim\sum\nolimits_{n,m\in\mathbb Z}a_{nm}e^{inx}e^{imy}.$$
Does there exist an integrable function $h:[0,2\pi]^2\to\mathbb R$ with
$$h\sim-\sum\nolimits_{n,m\in\mathbb Z}m n a_{nm}e^{inx}e^{imy}?$$
\end{prob}

\end{document}